\documentclass[10pt]{amsart}
\usepackage{amssymb}
\usepackage{dsfont}
\usepackage{amscd}
\usepackage[mathscr]{euscript} 
\usepackage[all]{xy}
\usepackage{url}
\usepackage{stmaryrd}
\usepackage{comment}
\usepackage[retainorgcmds]{IEEEtrantools}
\usepackage{hyperref}
\usepackage{graphicx}
\usepackage{mathtools}

\title{Witt vectors and separably closed fields with higher derivations}

\author[D.M. HOFFMANN]{Daniel Max Hoffmann$^{\dagger}$}
\thanks{$^{\dagger}$SDG.
The author is supported by the National Science Centre (Narodowe Centrum Nauki, Poland) 
grants no. 2016/20/T/ST1/00482, 
2016/21/N/ST1/01465, and 2015/19/B/ST1/01150.}
\address{Daniel Max Hoffmann, Instytut Matematyki\\
Uniwersytet Warszawski\\
Warszawa\\
Poland}
\email{daniel.max.hoffmann@gmail.com}
\urladdr{{https://sites.google.com/site/danielmaxhoffmann/}}

  \DeclareMathOperator{\id}{id}
 \DeclareMathOperator{\fr}{Fr}

\DeclareMathOperator{\ch}{char}

\DeclareMathOperator{\ev}{ev}

\DeclareMathOperator{\sep}{sep}

\DeclareMathOperator{\sch}{SCH}
\DeclareMathOperator{\ddf}{DF}\DeclareMathOperator{\dcf}{DCF}\DeclareMathOperator{\scf}{SCF}
\DeclareMathOperator{\shf}{SHF}
\newtheorem{theorem}{Theorem}[section]
\newtheorem{prop}[theorem]{Proposition}
\newtheorem{lemma}[theorem]{Lemma}
\newtheorem{cor}[theorem]{Corollary}
\newtheorem{fact}[theorem]{Fact}
\theoremstyle{definition}
\newtheorem{definition}[theorem]{Definition}
\newtheorem{example}[theorem]{Example}
\newtheorem{remark}[theorem]{Remark}

\theoremstyle{remark}

\DeclareMathOperator{\ve}{V}
\DeclareMathOperator{\re}{R}

\makeatletter
\providecommand*{\cupdot}{%
  \mathbin{%
    \mathpalette\@cupdot{}%
  }%
}
\newcommand*{\@cupdot}[2]{%
  \ooalign{%
    $\m@th#1\cup$\cr
    \hidewidth$\m@th#1\cdot$\hidewidth
  }%
}
\makeatother

\begin{document}

\newcommand{\twoc}[3]{ {#1} \choose {{#2}|{#3}}}
\newcommand{\thrc}[4]{ {#1} \choose {{#2}|{#3}|{#4}}}
\newcommand{\Zz}{{\mathds{Z}}}
\newcommand{\Ff}{{\mathds{F}}}
\newcommand{\Cc}{{\mathds{C}}}
\newcommand{\Rr}{{\mathds{R}}}
\newcommand{\Nn}{{\mathds{N}}}
\newcommand{\Qq}{{\mathds{Q}}}
\newcommand{\Kk}{{\mathds{K}}}
\newcommand{\Pp}{{\mathds{P}}}
\newcommand{\ddd}{\mathrm{d}}
\newcommand{\Aa}{\mathds{A}}
\newcommand{\dlog}{\mathrm{ld}}
\newcommand{\ga}{\mathbb{G}_{\rm{a}}}
\newcommand{\gm}{\mathbb{G}_{\rm{m}}}
\newcommand{\gaf}{\widehat{\mathbb{G}}_{\rm{a}}}
\newcommand{\gmf}{\widehat{\mathbb{G}}_{\rm{m}}}
\newcommand{\gdf}{\mathfrak{g}-\ddf}
\newcommand{\gdcf}{\mathfrak{g}-\dcf}
\newcommand{\fdf}{F-\ddf}
\newcommand{\fdcf}{F-\dcf}
\newcommand{\mw}{\scf_{\text{MW},e}}

\maketitle
\begin{abstract}
The main scope of this short paper is to provide a modification 
of the axioms given by Messmer and Wood
for the theory of separably closed fields of positive characteristic and finite imperfectness degree.
The original axioms failed to meet natural expectations,
and therefore a new axiomatization was given (i.e. Ziegler's one),
but the new axioms do not follow
the initial idea of Messmer and Wood.
Therefore, 
we aim to give a correct axiomatization which is more similar to the original one and which,
as the original axioms, involves only one Hasse-Schmidt derivation, this time 
based on the iterativity conditions corresponding to the Witt group.
\end{abstract}

\section{Introduction}
Messmer and Wood proposed in \cite{MW}  an axiomatization of the theory of separably closed fields of positive characteristic $p$ and finite imperfectness degree $e$, denoted by $\scf_{p,e}$. 
Their primary aim was to describe models of $\scf_{p,e}$ 
in a language without naming a $p$-basis, but with one higher derivation (then a \emph{basis} could be reconstructed via derivations).
Authors expected
that $\scf_{p,e}$ in this new set-up will enjoy
quantifier elimination and elimination of imaginaries.
This was related to Hrushovski's celebrated proof of the relative Mordell-Lang conjecture (\cite{Hr7}).
More precisely, if models of $\scf_{p,e}$ can be described in a differential set-up having
quantifier elimination and elimination of imaginaries, 
then a uniformed version of Hrushovski's proof, working for all characteristics, might be derived.

Unfortunately, as Ziegler pointed out in \cite{Zieg3}, 
the axiomatization of Messmer and Wood fails
since there are some gaps and false claims in \cite{MW}.
Besides of that, we found out some more conceptual problems
in the axioms given in \cite{MW} - we give more details in Section \ref{original}. 
Moreover, in \cite{Zieg3}, Ziegler provided
a very natural axiomatization of $\scf_{p,e}$
which involves higher derivations.
However, Ziegler's idea
is quite different from the idea of Messmer and Wood, since Ziegler uses a collection of $e$ higher derivations instead of one single higher derivation.
In this paper, we propose a language and axioms which are
more in the spirit of \cite{MW}, for example we consider one sequence of $p^e$-nilpotent operators
instead of $e$ sequences of $p$-nilpotent operators (as in \cite{Zieg3}).
The main difficulty lies in finding a proper iterativity conditions and a group staying behind them.
In general, one could consider a theory of $F$-fields (see Definition \ref{f_ring}), 
where $F$ is an $e$-dimensional formal group law.
Both descriptions of the theory $\scf_{p,e}$, our and Ziegler's,
fit in with this general frame.
Ziegler's axioms correspond to the existentially closed $\mathbb{G}_a^e$-fields, our axioms correspond to the existentially closed
$W_e$-fields, where $W_e$ denotes the Witt group of dimension $e$.
The special property of the Witt group $W_e$ which is crucial for us, is that any $e$-dimensional Hasse-Schmidt derivation obeying the iterativity conditions coming from $W_e$ can be reconstructed from some 1-dimensional Hasse-Schmidt derivation (see Lemma \ref{shf_We}). 
Thus our strategy in this paper is to compress an $e$-dimensional Hasse-Schmidt derivation into a 1-dimensional Hasse-Schmidt derivation.
This is not the case if we choose the iterativity conditions coming from $\mathbb{G}_a^e$, unless $e=1$.

The fundamental theory, on which everything is build is the theory of existentially closed $F$-fields and we regard this paper as an application of the general theory of existentially closed $F$-fields, which was studied in \cite{HK}.

It should be stressed, that we do not aim to give a ``better" axiomatization than the one from \cite{Zieg3}, but the one which is the closest one to the original idea from \cite{MW}. So, this paper may be considered as an end note to \cite{MW}.

I thank Piotr Kowalski for the main idea of this paper and for many helpful comments on every stage of my work. I am grateful to the referee for careful reading and very useful remarks.

\section{Basic notions about Hasse-Schmidt derivations}
By $f^{(n)}:S\to S$, where $S$ is a set, we denote
the composition of $f$ with itself $n$ times.
For the rest of this paper, we
fix a prime number $p$ and a positive natural number $e$. 
Let $R$ and $S$ be $\mathbb{F}_p$-algebras.
The symbol $\bar{x}^{\mathbf{i}}$, where $\mathbf{i}=(i_1,\ldots,i_e)\in\mathbb{N}^e$
and $\bar{x}=(x_1,\ldots,x_e)$ is a tuple of elements from $R$,
stands for the element $x_1^{i_1}\cdot\ldots\cdot x_e^{i_e}$.
We use similar convention for variables.
For any $n\in\mathbb{N}$, $[n]$ denotes the set $\lbrace 0,\ldots, n-1\rbrace$. 

For the convenience of the reader, we collect below definitions of all the languages, which are introduced in this paper. The symbols $b_1,\ldots,b_e$ below are constant symbols:
\begin{IEEEeqnarray*}{rClCrCl}
 \mathcal{L}^0 &:=& \lbrace+,-,\cdot,0,1\rbrace, & &\mathcal{L}^0_B &:=& \mathcal{L}^0\cup\lbrace b_1,\ldots,b_e\rbrace, \\
 \mathcal{L} &:=& \mathcal{L}^0\cup\lbrace D_{n}\rbrace_{n\in\mathbb{N}}, &\qquad &
  \mathcal{L}_B &:=&\mathcal{L}\cup\lbrace b_1,\ldots,b_e\rbrace, \\
 \mathcal{L}^{\ast} &:=& \mathcal{L}^0\cup\lbrace D_{\mathbf{i}}\rbrace_{\mathbf{i}\in\mathbb{N}^e}, & & \mathcal{L}^{\ast}_{B} &:=& \mathcal{L}^{\ast}\cup\lbrace b_1,\ldots,b_e\rbrace.
\end{IEEEeqnarray*}
Moreover, we refer sometimes to the language of the theory from \cite{MW}:
$\mathcal{L}_{MW}=\mathcal{L}^0\cup\lbrace D_{p^i}\rbrace_{i\in\mathbb{N}}$, which is a sublanguage of $\mathcal{L}$.

\begin{remark}
If $S$ is a complete local $R$-algebra and elements $s_1,\ldots,s_n$ belong to the maximal ideal of $S$, then there exists (by Theorem 7.16 in \cite{EIS}) a unique $R$-algebra homomorphism $R\llbracket X_1,\ldots,X_n\rrbracket\to S$
sending each $X_i$ to $s_i$, which we denote by $\ev_{(s_1,\ldots,s_n)}$. 
For example $\ev_{\bar{0}}:R\llbracket X_1,\ldots,X_n\rrbracket\to R$ sends
any $F(X_1,\ldots,X_n)\in R\llbracket X_1,\ldots,X_n\rrbracket$ to the element $\ev_{\bar{0}}(F)$ denoted by $F(0,\ldots,0)$.
\end{remark}

\noindent
A $k$-algebra homomorphism
$$\mathbb{D}:R\to R\llbracket \bar{X}\rrbracket,$$
where $\bar{X}=(X_1,\ldots,X_e)$, is called a (\emph{an $e$-dimensional}) \emph{Hasse-Schmidt derivation} 
if $\ev_{\bar{0}}\circ\mathbb{D}=\id_R$. 
Equivalently, a Hasse-Schmidt derivation is
a collection $\mathbb{D}=(D_{\mathbf{i}}:R\to R)_{\mathbf{i}\in\mathbb{N}^e}$
of $k$-linear maps
such that $D_{\bar{0}}=\id_R$ and 
$$D_{\mathbf{i}}(rs)=\sum\limits_{\mathbf{j}+\mathbf{k}=\mathbf{i}}D_{\mathbf{j}}(r)D_{\mathbf{k}}(s)$$
holds for every $r,s\in R$.

Let $F(\bar{X},\bar{Y})\in (\mathbb{F}_p\llbracket\bar{X},\bar{Y}\rrbracket)^e$ be an \emph{$e$-dimensional formal group law} (for the definition see \cite[Section 9.1]{Hazew}).
\begin{definition}\label{f_der}
 We call a Hasse-Schmidt derivation $\mathbb{D}$ $\;F$-\emph{iterative}
 if the following diagram commutes
\begin{equation*}
 \xymatrixcolsep{4.5pc}\xymatrixrowsep{1.5pc}\xymatrix{ 
  R \ar[r]^{\mathbb{D}_{\bar{X}}} \ar[d]_{\mathbb{D}_{\bar{X}}} & R\llbracket\bar{X}\rrbracket \ar[d]^{\mathbb{D}_{\bar{Y}}\llbracket\bar{X}\rrbracket}
  \\ R\llbracket\bar{X}\rrbracket  \ar[r]_{\ev_F}   & R\llbracket\bar{X},\bar{Y}\rrbracket},
\end{equation*}
where $\mathbb{D}_{\bar{X}}:=\mathbb{D}$, 
$\mathbb{D}_{\bar{Y}}:=\ev_{\bar{Y}}\circ\mathbb{D}$
and $\mathbb{D}_{\bar{Y}}\llbracket\bar{X}\rrbracket$ is given by
$$\mathbb{D}_{\bar{Y}}\llbracket\bar{X}\rrbracket
\Big(\sum\limits_{\mathbf{i}\in\mathbb{N}^e}r_{\mathbf{i}}\bar{X}^{\mathbf{i}}\Big)= 
\sum\limits_{\mathbf{i}\in\mathbb{N}^e} 
\mathbb{D}_{\bar{Y}}\big(r_{\mathbf{i}}\big)\bar{X}^{\mathbf{i}}.$$
We write briefly ``$F$-derivation" instead of ``$F$-iterative Hasse-Schmidt derivation''.
\end{definition}

\begin{definition}\label{f_ring}
 We call a pair $(R,\mathbb{D})$  an $F$-\emph{ring} if $\mathbb{D}$ is an $F$-derivation on $R$. 
 A $k$-algebra homomorphism $f:R\to S$ between $F$-rings
 $(R,\mathbb{D})$ and $(S,\mathbb{D}')$ is an $F$-\emph{morphism} (a morphism of $F$-rings)
 if $f\circ D_{\mathbf{i}}=D'_{\mathbf{i}}\circ f$ for all $\mathbf{i}\in\mathbb{N}^e$. 
 In similar manner, one may define $F$-\emph{fields}.
\end{definition}

\noindent
Commutativity of the diagram from Definition \ref{f_der} can be equivalently expressed by
so called \emph{iterativity conditions}:
$$ D_{\mathbf{i}}\circ D_{\mathbf{j}}=\sum\limits_{\mathbf{l}\in\mathbb{N}^e}\alpha_{\mathbf{i},\mathbf{j}}(\mathbf{l})\cdot D_{\mathbf{l}},$$
where $\mathbf{i},\mathbf{j}\in\mathbb{N}^e$ and
the constants $\alpha_{\mathbf{i},\mathbf{j}}(\mathbf{l})\in \mathbb{F}_p$ are given by $F$. 
For example, the standard iterativity, which
corresponds to the formal group law of $\mathbb{G}_a^e$, is given by
$$D_{\mathbf{i}}\circ D_{\mathbf{j}}={i_1+j_1\choose i_1}\cdot\ldots\cdot{i_e+j_e\choose i_e} D_{\mathbf{i}+\mathbf{j}},$$
where $\mathbf{i}=(i_1,\ldots,i_e),\mathbf{j}=(j_1,\ldots,j_e)\in\mathbb{N}^e$. The reader may consult \cite{HK} for more details
on the iterative higher dimensional Hasse-Schmidt derivations. 

For any $N\geqslant 1$, the ``multiplication by $N$ map'' is defined inductively as follows:
$$[1]_F:=\bar{X},$$
$$[N+1]:=F(\bar{X},[N]_F).$$
\begin{fact}[Corollary 2.25 in \cite{Hoff1}]\label{2.25}
 Assume that $(R,\mathbb{D})$ is a $F$-ring, 
 then for any $r\in R$ we have that
 $$\sum\limits_{\mathbf{i}\in\mathbb{N}^e}D_{\mathbf{i}}^{(p)}(r)\bar{X}^{\mathbf{i}}=
 \ev_{[p]_F(\bar{X}^{1/p})}\big(\sum\limits_{\mathbf{i}\in\mathbb{N}^e}D_{\mathbf{i}}(r)\bar{X}^{\mathbf{i}}\big) .$$
\end{fact}

The reader may check basic notions about the Witt groups in \cite[p. 172.]{serre1988algebraic} and \cite{Witt}. The Witt group (over $\mathbb{F}_p$) of dimension $n$ is denoted by $W_n$. 
Let us recall a few well-known homomorphisms related to the group structure on $W_n$: 
\begin{itemize}
 \item the Frobenius $\fr:W_n\to W_n$, given by $(x_1,\ldots,x_n)\mapsto (x_1^p,\ldots,x_n^p)$,
 \item the Verschiebung $\ve:W_n\to W_{n+1}$, given by $(x_1,\ldots,x_n)\mapsto(0,x_1,\ldots,x_n)$,
 \item the restriction $\re:W_n\to W_{n-1}$, given by $(x_1,\ldots,x_n)\mapsto(x_1,\ldots,x_{n-1})$.
\end{itemize}
For us, the most important fact is that the product $\fr\circ\ve\circ\re$ is equal to the multiplication by $p$ in the algebraic group $W_n$. 
We assume, for the rest of this paper, that $H(\bar{X},\bar{Y})\in(\mathbb{F}_p[\bar{X},\bar{Y}])^e$ defines
the group law on $W_e$. Instead of $H$-iterativity, $H$-derivations, $H$-rings etc., we write $W_e$-iterativity, $W_e$-derivations, $W_e$-rings etc. 

\begin{example}\label{ex:1}
For $p=3$ and $e=2$, the group law in $W_e$ is given by:
$$(X_1,X_2)\ast(Y_1,Y_2)=(X_1+Y_1,X_2+Y_2+X_1^2Y_1+X_1Y_1^2).$$
\end{example}

\begin{definition}
By $W_e-\dcf$ we denote an $\mathcal{L}^{\ast}$-theory, which models are pairs
$(K,\mathbb{D})$ satisfying
\begin{itemize}
\item the field $K$ is separably closed $W_e$-field of characteristic $p$,
\item we have $[K:K^p]=p^e$ (the degree of imperfection is equal to $e$),
\item we have $\ker D_{(1,0,\ldots,0)}\cap\ldots\cap\ker D_{(0,\ldots,0,1)}=K^p$
($\mathbb{D}$ is a \emph{strict} $W_e$-derivation).
\end{itemize}
\end{definition}
Properties of this theory were already described in \cite{HK},
as a particular case of the more general theory $F-\dcf$, where $F$ is an arbitrary formal group law.
We summarise these properties in the fact below.
\begin{fact}\label{propertiesWe}
  Theory $W_e-\dcf$ (in the language $\mathcal{L}^{\ast}$)
is stable, complete, has elimination of imaginaries and quantifier elimination.
\end{fact}

\section{The original theory}\label{original}
In this section, we recall the theory introduced in \cite{MW}.
The main aim of Messmer and Wood was to obtain 
an axiomatization for separably closed fields in
the language of rings expanded by symbols for one higher derivation. 
It was pointed out in \cite{Zieg3} (on the request of Messmer and Wood) that there are some gaps in \cite{MW}. However, the problem with the theory proposed by Messmer and Wood is more conceptual and we will provide some details on that.


We are working in the language $\mathcal{L}_{MW}$.
For $n=\gamma_0+\gamma_1 p+\ldots+\gamma_s p^s$ ($p$-adic expansion), where $s\in\mathbb{N}$, $0\leqslant \gamma_0,\ldots,\gamma_s<p$, 
consider a new function symbol
\begin{equation}\label{D_n}
D_n=\frac{(p!)^{\gamma_1}\cdot\ldots\cdot(p^s!)^{\gamma_s}}{n!}D_1^{(\gamma_0)}\circ D_p^{(\gamma_1)}\circ\ldots\circ D_{p^s}^{(\gamma_s)},\qquad D_0=\id
\end{equation}
(note that the coefficient $\frac{(p!)^{\gamma_1}\cdot\ldots\cdot(p^s!)^{\gamma_s}}{n!}$ is not divisible by $p$).
\begin{definition}\label{shf_org}
Recall that the theory $\shf_{p,e}$ from \cite{MW} (in the language $\mathcal{L}_{MW}$) is given by:
 \begin{enumerate}
  \item[H0] the axioms for fields of characteristic $p$,
  \item[H1] for all $i\in\mathbb{N}$: $D_{p^i}(x+y)=D_{p^i}(x)+D_{p^i}(y)$,
  \item[H2] for all $i\in\mathbb{N}$: $D_{p^i}(x\cdot y)=\sum\limits_{k+l=p^i}D_k(x)\cdot D_l(y)$,
  \item[H3] for all $i,j\in\mathbb{N}$: $D_{p^i}(D_{p^j}(x))=D_{p^j}(D_{p^i}(x))$,
  \item[H4] for all $i\in\mathbb{N}$: $D_{p^i}^{p^e}(x)=0$,
  \item[H5] $(\exists x)\,D_1^{p^e-1}(x)\neq 0$,
  \item[H6] $D_1(x)=0\rightarrow (\exists y)\,x=y^p$ (strictness),
  \item[H7] the axioms for separably closed fields.
 \end{enumerate}
 \end{definition}
\begin{example}
Assume that $p=e=2$ and let $(K,(D_{2^i})_{i\in\mathbb{N}})$ be a model for H0-H5.
We have $D_3=D_1D_2$, moreover there exists $x\in K$ such that $D_1^{(3)}(x)\neq 0$ and $D_1^{(4)}(x)=0$.
Let $y=D_1^{(2)}(x)$, it is natural to expect that
$$D_3(xy)=D_3(x)y+D_2(x)D_1(y)+D_1(x)D_2(y)+xD_3(y),$$
in other words $D_3$ should satisfy the generalized Leibniz rule.
Unfortunately
$$D_3(xy)=D_1\big(D_2(xy)\big)=D_3(x)y+D_2(x)D_1(y)+D_1(x)D_2(y)+xD_3(y)+$$
$$D_1\big(D_1(x)\big)D_1(y)+D_1(x)D_1\big(D_1(y)\big),$$
and
$$D_1D_1(x)D_1(y)+D_1(x)D_1D_1(y)=D_1^{(2)}(x)\cdot D_1^{(3)}(x)\neq 0.$$
\end{example}
Therefore the sequence $\lbrace D_n\rbrace_{n\in\mathbb{N}}$ is not a \emph{higher derivation} (see \cite{Mats1}), so the theory $\shf_{p,e}$ does not describe
a theory of fields with derivations in a manner that the authors of \cite{MW} expected. Actually, we do not even know whether the theory $\shf_{p,e}$
is consistent.

\section{The theory $\shf_{p,e}$ revised}
\subsection{The axiomatization}
The definition of the operators from (\ref{D_n}) includes some information about the standard iterativity, which leads to problems described in
Section \ref{original}.
Instead of defining operators by formulas (\ref{D_n}), we add to the language new symbols for those operators, and so replace the language with $\mathcal{L}$. It is a minor, but important, change. 
We present a smooth modification of the original theory from \cite{MW} in our language.

\begin{definition}\label{shf_def}
 $\shf_{p,e}'$ denotes the $\mathcal{L}$-theory which contains the following axioms:
 \begin{enumerate}
  \item[H0$'$] the axioms for fields of characteristic $p$,
  \item[H1$'$] for all $n\in\mathbb{N}$: $D_n(x+y)=D_n(x)+D_n(y)$,
  \item[H2$'$] for all $n\in\mathbb{N}$: $D_n(x\cdot y)=\sum\limits_{k+l=n}D_k(x)\cdot D_l(y)$,
  \item[H3$'$] for all $i,j\in\mathbb{N}$: $D_i(D_j(x))=D_j(D_i(x))$,
  \item[H4$'$] for all $i\in\mathbb{N}$: $D_i^{p^e}(x)=0$,
  \item[H5$'$] $(\exists x)\,D_1^{p^e-1}(x)\neq 0$,
  \item[H6$'$] $D_1(x)=0\rightarrow (\exists y)\,x=y^p$ (strictness),
  \item[H7$'$] the axioms for separably closed fields.
 \end{enumerate}
\end{definition}

\noindent
\begin{remark}
\begin{enumerate}
\item  Axiom scheme H3$'$ is negligible (we do not use H3$'$ in the below proofs and it will be implied 
by the axiom H8$'$, which will be defined later),
but we keep it, because
in this form our axiomatization contains - in some sense - the original
axioms for $\shf_{p,e}$ given in Definition \ref{shf_org}.
 
\item  The axioms included in H0$'$, H4$'$-H7$'$ are the same as those included in H0, H4-H7 respectively. Axiom schemes H1 and H3 from the original theory are simply
special cases of axiom schemes H1$'$ and H3$'$ from our modification.
 
 \item  Axiom scheme H2 from the original theory, if expressed using only symbols from $\mathcal{L}_{MW}$,
contains some information about the standard iterativity
(again - which is actually the main source of problems).
The new one is independent from any iterativity rules. However, both the old H2 and the new H2$'$ axiom schemes
code the generalized Leibniz rule.
 
\item By Corollary \ref{consistency}, we are sure that the theory $\shf_{p,e}'$ is consistent, what is still unclear for the theory $\shf_{p,e}$.
\end{enumerate}
\end{remark}

The theory $\shf_{p,e}'$ is not complete, for example one could consider
$e=1$, $p=2$ and a sentence
$(\forall x) D_1D_2(x)=D_3(x)$, which is true in a model of $\mathbb{G}_a-\dcf$ (so also a model of 
$\shf_{2,1}'$), but which does not hold in the model of $\shf_{2,1}'$ described in the following example.

\begin{example}\label{ex:incompleteness}
Assume that $e=1$ and $p=2$.	
Let $\mathbb{D}:\mathbb{F}_p[t]\to \mathbb{F}_p[t]\llbracket X\rrbracket$ be the $\mathbb{F}_p$-algebra homomorphism determined by $t\mapsto t+X+X^2+X^3+\ldots$. Since $\mathbb{D}$ is a Hasse-Schmidt derivation on $\mathbb{F}_p[t]$, we can extend it to the field of fractions of $\mathbb{F}_p[t]$ and then to the separable closure (by Proposition 3.3. from \cite{HK}), say that $(F,\mathbb{D})$ is the described extension to $F:=(\mathbb{F}_p(t))^{\sep}$. In an obvious way, $(F,\mathbb{D})$ satisfies
H0$'$, H1$'$, H2$'$, H5$'$ and H7$'$. Because $(\mathbb{F}_p[t],\mathbb{D})$ is strict (i.e. it satisfies H6$'$) and has finite imperfectness degree ($e=1$), any \'{e}tale extension is strict,
thus $(F,\mathbb{D})$ is strict, i.e. it satisfies H6$'$.
To see that H3$'$ and H4$'$ hold, we need to use methods from Section 3.3 from \cite{HK1}.
Let us see how we can get H3$'$. Using the notation from \cite{HK1}, first, we show that
$\ev_{(X_2,X_1)}E_2E_1=E_2E_1$ is true over $\mathbb{F}_p[t]$, thus $(\mathbb{F}_p[t],\mathbb{D})$ satisfies H3$'$. Because of the uniqueness from Proposition 3.3 from \cite{HK} and \'{e}tality of $\mathbb{F}_p[t]\subseteq F$,
we see that two-dimensional Hasse-Schmidt derivations $E_2E_1$ and $\ev_{(X_2,X_1)}E_2E_1$, which
extend $E_2E_1|_{\mathbb{F}_p[t]}$, must coincide. Thus H3$'$ holds in $(F,\mathbb{D})$.
For H4$'$, note that $\ev_{(X,X)}E_2E_1(t)=t$.
By Lemma 3.7 from \cite{HK1}, it is $D^{(2)}_i|_{\mathbb{F}_p[t]}$ for any $i>0$.
Then, again, the uniqueness from Proposition 3.3 from \cite{HK}, this time used for 
$\ev_{(X,X)}E_2E_1$ (the $2$-th compositions), gives us that also $D^{(2)}_i(a)=0$ for any $a\in F$ and $i>0$.
Therefore $(F,\mathbb{D})\models \shf_{p,e}'$ and we can calculate that
$D_1D_2(t)=0$ and $D_3(t)=1$.
\end{example}

Besides of the above reasons for incompleteness, there is still some ``free space" for the choice of iterativity conditions.
In other words, rejecting the formulas from (\ref{D_n}) gives us freedom to choose some iterativity rule.
However, iterativity conditions coming from $\mathbb{G}_a$ will correspond to separably closed fields with the degree of imperfection equal to $1$ (i.e. $e=1$) and this is not our goal. 
Therefore the theory $\shf_{p,e}'$ should be extended by iterativity conditions coming from a different
formal group law
and it will be done by adding one more axiom scheme. 
Before describing this axiom scheme, we state a necessary fact.
Let $(K,\mathbb{D})$ be an $\mathcal{L}^{\ast}$-structure (e.g. a model of $W_e-\dcf$) and $\mathbb{D}=(D_{\mathbf{i}})_{\mathbf{i}\in\mathbb{N}^e}$. 
 We introduce
   $$\partial_{i,n}:=D_{(0,\ldots,0,\underbracket[0.5pt]{n}_{i\text{-th place}},0\ldots,0)},$$
 where $n\in\mathbb{N}$, $i\leqslant e$.
 Note that $(\partial_{i,n})_{n\in\mathbb{N}}$ is a Hasse-Schmidt derivation for every $i\leqslant e$.
Before stating the next lemma, we provide an example where one can see how the general methods from the proof of Lemma \ref{We_iter} work in a particular case.
 
\begin{example}\label{ex:2}
Let us calculate what is $D_{(i,0)}\circ D_{(0,j)}$ for $W_2$ with $p=3$ (see Example \ref{ex:1}).
We know that $H(X_1,0,0,Y_2)=(X_1,0)\ast(0,Y_2)=(X_1,Y_2)$. 
Let $(K,\mathbb{D})\models W_2-\dcf$, $\ch(K)=3$, and let $a\in K$ be arbitrary.
By the formula following from the diagram from Definition \ref{f_der}, i.e.
$$\sum\limits_{\mathbf{i},\mathbf{j}\in\mathbb{N}^e}D_{\mathbf{i}}D_{\mathbf{j}}(a)\bar{X}^{\mathbf{i}}\bar{Y}^{\mathbf{j}}=\sum\limits_{\mathbf{k}\in\mathbb{N}^e}D_{\mathbf{k}}(a)(H(\bar{X},\bar{Y}))^{\mathbf{k}},$$
where $\bar{X}=(X_1,X_2)$ and $\bar{Y}=(Y_1,Y_2)$, after putting $X_2=Y_1=0$, we get
$$\sum\limits_{i,j\in\mathbb{N}}D_{(i,0)}D_{(0,j)}(a)X_1^{i}Y_2^{j}=\sum\limits_{k,l\in\mathbb{N}}D_{(k,l)}(a)X_1^kY_2^l.$$
Hence $D_{(i,0)}\circ D_{(0,j)}=D_{(i,j)}$.
\end{example} 
 
\begin{lemma}\label{We_iter} 
Assume that $(K,\mathbb{D})\models W_e-\dcf$. We have the following.
 \begin{itemize}
  \item[i)] For every $\mathbf{i}\in\mathbb{N}^e\setminus\{\bar{0}\}$, it follows that $D_{\mathbf{i}}^{(p^e)}=0$.
  \item[ii)] We have
   $$\partial_{e,n}^{(p)}=0\;\;\text{ and }\;\; \partial_{i,n}^{(p)}=\partial_{i+1,n},$$
   where $n\in\mathbb{N}$ and $i<e$.
  \item[iii)] For any $i_1,\ldots,i_e\in\mathbb{N}$, it holds that 
  $$D_{(i_1,\ldots,i_e)}=\partial_{1,i_1}\circ\partial_{1,i_2}^{(p)}\circ\ldots\circ\partial_{1,i_e}^{(p^{e-1})}.$$
\end{itemize}
\end{lemma}
 
\begin{proof}
 Recall that $H\in (\mathbb{F}_p[\bar{X},\bar{Y}])^e$ defines the group law of $W_e$.
 We begin with proving the first item.
 It follows from Fact \ref{2.25} that, for every $a\in K$, we have
 $$\sum\limits_{\mathbf{i}\in\mathbb{N}^e}D_{\mathbf{i}}^{(p)}(a) \bar{X}^{\mathbf{i}}=\ev_{[p]_H(\bar{X}^{1/p})}
 \Big(\sum\limits_{\mathbf{i}\in\mathbb{N}^e}D_{\mathbf{i}}(a) \bar{X}^{\mathbf{i}}\Big).$$
Therefore $\mathbb{D}^{(p)}=(D_{\mathbf{i}}^{(p)})_{\mathbf{i}\in\mathbb{N}^e}$ is a Hasse-Schmidt derivation.
 But $\mathbb{D}^{(p)}$ is also $W_e$-iterative (relatively easy diagram chase as in \cite[Lemma 2.6]{HK1}) and we can use Fact \ref{2.25}
 for $\mathbb{D}^{(p)}$. Repeating this process we obtain
 $$\sum\limits_{\mathbf{i}\in\mathbb{N}^e}D_{\mathbf{i}}^{(p^e)}(a) \bar{X}^{\mathbf{i}}=\big(\ev_{[p]_H(\bar{X}^{1/p})}\big)^{(e)}
 \Big(\sum\limits_{\mathbf{i}\in\mathbb{N}^e}D_{\mathbf{i}}(a) \bar{X}^{\mathbf{i}}\Big).$$
 Because $(\fr\circ\ve\circ\re)^{(e)}(x_1,\ldots,x_e)=(0,\ldots,0)$, we see that
 $$\big(\ev_{[p]_H(\bar{X}^{1/p})}\big)^{(e)}=\ev_{\bar{0}},$$
 so $D_{\mathbf{i}}^{(p^e)}=0$ for any $\mathbf{i}\in\mathbb{N}^e\setminus\{\bar{0}\}$.
 \
 \\
 The second item follows from Fact \ref{2.25} for multiplication by $p$ in $W_e$ (coinciding with $\fr\circ\ve\circ\re$).
 The last item follows from the second item and the equality
 $$D_{(i_1,\ldots,i_n,0,\ldots,0)}\circ D_{(0,\ldots,0,i_{n+1},\ldots,i_e)}=D_{(i_1,\ldots,i_e)},$$
 where $n\leqslant e$. 
 To prove the above equality, as in Example \ref{ex:2}, we use the diagram from Definition \ref{f_der}, namely the following formula:
$$\sum\limits_{\mathbf{i},\mathbf{j}\in\mathbb{N}^e}D_{\mathbf{i}}D_{\mathbf{j}}(a)\bar{X}^{\mathbf{i}}\bar{Y}^{\mathbf{j}}=\sum\limits_{\mathbf{k}\in\mathbb{N}^e}D_{\mathbf{k}}(a)(H(\bar{X},\bar{Y}))^{\mathbf{k}},$$
where $a\in K$, and the following property of $W_e$:
\begin{equation}\label{witt_induction}
 (X_1,\ldots,X_n,0,\ldots,0)\ast(0,\ldots,0,X_{n+1},\ldots,X_e)=(X_1,\ldots,X_e),
\end{equation}
 where $n\leqslant e$, which can be obtained by induction involving the formulas (a) and (b),
 defining the group structure of the Witt group, from \cite[p. 128]{Witt}.
\end{proof}

We are ready now to define the promised axiom scheme H8$'$ - the axiom scheme, which corresponds to the iterativity conditions.
The main idea here is to ``compress" iterativity conditions of the $e$-dimensional formal group law $W_e$ into some iterativity conditions of an one-dimensional nature.
By Lemma \ref{We_iter}.iii), in any model $(K,\mathbb{D})$ of the theory $W_e-\dcf$ the iterativity conditions:
\begin{equation}\tag{$\mathbf{Y}_{\mathbf{i},\mathbf{j}}$}
 D_{\mathbf{i}}\circ D_{\mathbf{j}}=\sum\limits_{\mathbf{l}}\alpha_{\mathbf{i},\mathbf{j}}(\mathbf{l})\cdot D_{\mathbf{l}},
\end{equation}
where 
$\alpha_{\mathbf{i},\mathbf{j}}(\mathbf{l})\in\mathbb{F}_p$ are given by $W_e$, can be expressed using only
operators of the form $D_{(n,0\ldots,0)}$, $n\in\mathbb{N}$, i.e.:
\begin{equation}\tag{$\mathbf{Y}'_{\mathbf{i},\mathbf{j}}$}
 \partial_{1,i_1}\circ\ldots\circ\partial_{1,i_e}^{(p^{e-1})}\circ
 \partial_{1,j_1}\circ\ldots\circ\partial_{1,j_e}^{(p^{e-1})}=
 \sum\limits_{\mathbf{l}}\alpha_{\mathbf{i},\mathbf{j}}(\mathbf{l})\cdot
 \partial_{1,l_1}\circ\ldots\circ\partial_{1,l_e}^{(p^{e-1})}, 
\end{equation}
where $(i_1,\ldots,i_e)=\mathbf{i}$, $(j_1,\ldots,j_e)=\mathbf{j}$ and $(l_1,\ldots,l_e)=\mathbf{l}$.
To obtain axioms in the language $\mathcal{L}$, which code the $W_e$-iterativity conditions, we just make the following translation:
\begin{equation}\tag{$\mathbf{Y}^{\ast}_{\mathbf{i},\mathbf{j}}$}
 D_{i_1}\circ\ldots\circ D_{i_e}^{(p^{e-1})}\circ
 D_{j_1}\circ\ldots\circ D_{j_e}^{(p^{e-1})}=
 \sum\limits_{\mathbf{l}}\alpha_{\mathbf{i},\mathbf{j}}(\mathbf{l})\cdot
 D_{l_1}\circ\ldots\circ D_{l_e}^{(p^{e-1})}.
\end{equation}
Finally, we can define the axiom scheme we need:
\begin{itemize}
 \item[H8$'$] for every $\mathbf{i},\mathbf{j}\in\mathbb{N}^e$: $(\mathbf{Y}^{\ast}_{\mathbf{i},\mathbf{j}})$.
\end{itemize}
Note that the above axiom scheme plays a similar role to the formulas for $D_n$ in the original theory $\shf_{p,e}$ (\ref{D_n}), but carries more data.

\begin{definition}
 By $\overline{\shf}_{p,e}$ we denote the theory given by the axioms contained in H0$'$-H8$'$, in the language $\mathcal{L}$.
\end{definition}

\subsection{Extension by definitions}
To prove the expected properties of the theory $\overline{\shf}_{p,e}$, we will use the notion of \emph{extension by definitions} as described in \cite[Chapter 4.6]{SHOE}
to add new symbols for functions. For convenience of the reader we include basics about this procedure. Assume that $T$ is a theory
in a language $L$ and $\psi(\mathbf{x},y)$ is a formula in $L$, $\mathbf{x}=(x_1,\ldots,x_n)$. Suppose that
\begin{itemize}
 \item[i)] $T\vdash (\exists y)\psi(\mathbf{x},y)\;\;$ (existence condition),
 \item[ii)] $T\vdash (\,\psi(\mathbf{x},y)\wedge \psi(\mathbf{x},y')\,\rightarrow\, y=y'\,)\;\;$ (uniqueness condition).
\end{itemize}
We define $L^{\ast}$ as $L$ with added an extra new $n$-ary function symbol $f$ and form $T^{\ast}$, a theory in the language $L^{\ast}$,
by adding to $T$ a new axiom
$$f(\mathbf{x})=y \leftrightarrow \psi(\mathbf{x},y).$$
This axiom will be called \emph{defining axiom} for $f$.
There is no problem to add more function symbols and defining axioms,
and a theory obtained in this way will be called \emph{extension by definitions} of $T$. Assume that $T^{\ast}$ is an extension by definitions of the theory $T$.

Let $\xi$ be a formula in $L^{\ast}$. There is a method for translating the formula $\xi$
to a formula $\xi^r$ in $L$, by replacing each occurrence of $f(\mathbf{x})=y$ with $\psi(\mathbf{x},y)$ - for details we refer to \cite[p. 59]{SHOE}. The following is folklore, thus we provide it without a proof.

\begin{fact}\label{properties01}
 \begin{itemize}
  \item[i)] If the theory $T^{\ast}$ is stable, then the theory $T$ is stable as well.
  \item[ii)] If the theory $T^{\ast}$ is complete, then the theory $T$ is complete as well.
  \item[iii)] If the theory $T^{\ast}$ has elimination of imaginaries, then the theory $T$ has elimination of imaginaries as well.
 \end{itemize}
\end{fact}

Recall that $\mathcal{L}^{\ast}=\mathcal{L}^0\cup\lbrace D_{\mathbf{i}}\rbrace_{\mathbf{i}\in\mathbb{N}^e}$.
After identifying $D_n=D_{(n,0,\ldots,0)}$, for each $n\in\mathbb{N}$, we can treat $\mathcal{L}$ as a sublanguage of $\mathcal{L}^{\ast}$.
By $(\overline{\shf}_{p,e})^{\ast}$ we denote
the theory $\overline{\shf}_{p,e}$ in the language $\mathcal{L}^{\ast}$
with added defining axioms for every $\mathbf{i}=(i_1,\ldots,i_e)\in\mathbb{N}^e$:
\begin{equation}\label{axiom1} 
D_{(i_1,\ldots,i_e)}(x)=y\;\leftrightarrow\; D_{i_1}\circ D_{i_2}^{(p)}\circ\ldots\circ D_{i_e}^{(p^{e-1})}(x)=y.
\end{equation}
It is obvious that the $\mathcal{L}^{\ast}$-theory $(\overline{\shf}_{p,e})^{\ast}$ is an extension by definitions of the $\mathcal{L}$-theory $\overline{\shf}_{p,e}$.
Recall that $\mathcal{L}^{\ast}$ is also a language for the theory $W_e-\dcf$.

\begin{lemma}\label{shf_We}
Any model of $(\overline{\shf}_{p,e})^{\ast}$ is a model of $W_e-\dcf$.
\end{lemma}

\begin{proof}
 Let $(K,\mathbb{D})\models (\overline{\shf}_{p,e})^{\ast}$, $\mathbb{D}=(D_{\mathbf{i}})_{\mathbf{i}\in\mathbb{N}^e}$.
 Being separably closed field means the same in both contexts, so H0$'$ and H7$'$ assure us that $K$ is a separably closed field of characteristic $p$.
 Now we show that $\mathbb{D}$ is a strict $W_e$-derivation, i.e.:
 \begin{enumerate}
  \item[(a)] $D_{(0,\ldots,0)}(x)=x$,
  \item[(b)] $D_{\mathbf{i}}(x+y)=D_{\mathbf{i}}(x)+D_{\mathbf{i}}(y)$, for all $\mathbf{i}\in\mathbb{N}^e$,
  \item[(c)] $D_{\mathbf{i}}(xy)=\sum\limits_{\mathbf{k}+\mathbf{l}=\mathbf{i}}D_{\mathbf{k}}(x)D_{\mathbf{l}}(y)$, for all $\mathbf{i}\in\mathbb{N}^e$,
  \item[(d)] $D_{\mathbf{i}}D_{\mathbf{j}}(x)=\sum\limits_{\mathbf{l}\in\mathbb{N}^e}\alpha_{\mathbf{i},\mathbf{j}}(\mathbf{l})D_{\mathbf{l}}(x)$,
  for all $\mathbf{i},\mathbf{j}\in\mathbb{N}^e$,
  \item[(e)] $D_{(1,0,\ldots,0)}(x)=\ldots=D_{(0,\ldots,0,1)}(x)=0\rightarrow(\exists y)(x=y^p)$.
 \end{enumerate}
By the defining axioms, proof of items (a)-(b) is straightforward.
For the item (c) we note that, as $(D_i)_{i\in\mathbb{N}}$ is a commutative Hasse-Schmidt derivation,
multiple use of Corollary 3.8 from \cite{HK1} gives us that each $(D_i^{(p^k)})_{i\in\mathbb{N}}$, where $k<e$, is a Hasse-Schmidt derivation. Thus $(D_{\mathbf{i}})_{\mathbf{i}\in\mathbb{N}^e}$, as a composition of $e$-many (1-dimensional) Hasse-Schmidt derivations is a Hasse-Schmidt derivation (cf. Remark 3.2(4) in \cite{HK}), in particular item (c), and also items (a) and (b) hold.
The item (d) is exactly coded by axioms from H8$'$.
Axiom H6$'$ contains the item (e).
\
\\
Now we are going to prove $[K:K^p]=p^e$. By the strictness, it is enough to show that $[K:C_K]=p^e$ for $C_K:=\ker D_{(1,0,\ldots,0)}\cap\ldots\cap\ker D_{(0,\ldots,0,1)}$.
By \cite[Corollary 3.21]{HK}, $[K:C_K]\leqslant p^e$, and by \cite[Corollary 2.2]{MW} (which remains valid), H4$'$ and H5$'$ imply $[K:C_K]\geqslant p^e$.
\end{proof}

\begin{lemma}\label{We_shf}
 Any model of $W_e-\dcf$ is a model of $(\overline{\shf}_{p,e})^{\ast}$.
\end{lemma}

\begin{proof}
  Let $(K,\mathbb{D})\models W_e$-$\dcf$. From Lemma \ref{We_iter}.iii), it follows that the defining axioms (\ref{axiom1}) are satisfied for each $\mathbf{i}\in\mathbb{N}^e$.
  Axioms from H8$'$ are modelled on the $W_e$-iterativity, so they are also satisfied. 
  For the occurrence of the axiom scheme H4$'$, we expect that $D_{(n,0,\ldots,0)}^{(p^e)}=0$ for every $n\in\mathbb{N}$, and Lemma \ref{We_iter}.i) implies this.
  Axiom scheme H2$'$ is fulfilled in a natural way.
  Only the axioms H5$'$ and H6$'$ need some argumentation.
 \
 \\
 \textbf{(H5$'$)}
 For each $i\leqslant e$ let 
 $$\partial_i:=D_{(0,\ldots,0,\underbracket[0.5pt]{1}_{i\text{-th place}},0\ldots,0)}.$$
By Lemma \ref{We_iter}.ii), we see that 
 $$\partial_i^{(p)}=\partial_{i+1},$$
 for $i<e$, and $\partial_e^{(p)}=0$. Let $F_{e+1}:=K$ and for every $i\leqslant e$ let 
 $$F_i:=\bigcap\limits_{j\geqslant i}\ker \partial_j.$$
 In this notation, $\partial_i^*:=\partial_i|_{F_{i+1}}$ is a derivation over $F_i$ satisfying $\partial_i^{\ast(p)}=0$ (the last thing implies
 the additive iterativity for $\partial_i^*$).
 By \cite[Corollary 3.21]{HK} for $\partial_i^*$, we have $[F_{i+1}:F_i]\leqslant p$. The axioms of the theory $W_e$-$\dcf$ assure us that $[F_{e+1}:F_1]=p^e$, 
 and so $[F_{i+1}:F_i]=p$ for each $i\leqslant e$. By \cite[Theorem 27.3]{mat}, for every $i\leqslant e$ there exists an element $x_i \in F_{i+1}\backslash F_i$
 such that $\partial_i(x_i)=1$.
 \
 \\
 We put $x:=x_1^{p-1}\cdot\ldots\cdot x_e^{p-1}$ and then
  \begin{IEEEeqnarray*}{rCl}
  D_1^{(p^e-1)}(x) &=& D_1^{((p-1)+p(p-1)+\ldots+p^{e-2}(p-1)+p^{e-1}(p-1))}(x) \\
                 &=& \big(\partial_1\big)^{(p-1)}\circ\big(\partial_1^{(p)}\big)^{(p-1)}\circ\ldots\circ\big(\partial_1^{(p^{e-1})}\big)^{(p-1)}(x) \\
                 &=& \big(\partial_1\big)^{(p-1)}\circ\big(\partial_2\big)^{(p-1)}\circ\ldots\circ\big(\partial_e\big)^{(p-1)}(x) \\
                 &=& \partial_1^{(p-1)}\Big(x_1^{p-1}\cdot\partial_{2}^{(p-1)}\big(\ldots x_{e-2}^{p-1}\cdot\partial_{e-1}^{(p-1)}\big(x_{e-1}^{p-1}
                 \cdot\partial_e^{(p-1)}(x_e^{p-1})\big)\ldots\big)\Big)\\
                 &=& \big((p-1)!\big)^e\neq 0
\end{IEEEeqnarray*}

 \
 \\
 \textbf{(H6$'$)}
 Axioms of the theory $W_e$-$\dcf$ imply that $F_1=K^p$, hence it is enough to show that $F_1=\ker\partial_1$. Obviously $F_1\subseteq\ker\partial_1$, and
 if $\partial_1(c)=0$, then also $\partial_i(c)=\partial_1^{(p^{i-1})}(c)=0$ for any $i\leqslant e$. Therefore $F_1\supseteq\ker\partial_1$.
\end{proof}

\begin{cor}\label{consistency}
\begin{enumerate}
\item 
The theory $(\overline{\shf}_{p,e})^{\ast}$ is equal (as the set of logical consequences) to the theory $W_e-\dcf$.
The theories $(\overline{\shf}_{p,e})^{\ast}$, $\overline{\shf}_{p,e}$ and $\shf_{p,e}'$ are consistent.

\item
The theory $(\overline{\shf}_{p,e})^{\ast}$ (in the language $\mathcal{L}^{\ast}$) is stable, complete, has elimination of imaginaries and quantifier elimination.

\item 
 Theory $\overline{\shf}_{p,e}$ (in the language $\mathcal{L}$):
 \begin{itemize}
  \item[i)] is stable,
  \item[ii)] is complete,
  \item[iii)] has elimination of imaginaries,
  \item[iv)] has quantifier elimination.
 \end{itemize}
\end{enumerate}
\end{cor}

Consider the theory of separably closed fields of characteristic $p$
and degree of imperfection $e$ in the language $\mathcal{L}^0$, denoted by $\scf_{p,e}$.
So far, we dealt with arrows in the following diagram:
\begin{equation*}
 \xymatrixcolsep{3.5pc}\xymatrixrowsep{3.5pc}\xymatrix{& & W_e-\dcf \ar[d]^{\text{reduct}} \ar@/^/[dl]^{\text{reduct}}
\\ \shf_{p,e}' \ar@{--}[r]_(.4){+\text{H8}'} & \overline{\shf}_{p,e} \ar[r]_{\text{reduct}} \ar@/^1.5pc/[ur]^{\text{ext. by def.}}& \scf_{p,e}}
\end{equation*}
(the reduction from $W_e-\dcf$ to $\scf_{p,e}$ is trivial, the reduction from $\overline{\shf}_{p,e}$ to $\scf_{p,e}$ needs \cite[Corollary 2.2]{MW}, which remains valid).
However there is also a path leading from $\scf_{p,e}$ to $W_e-\dcf$. 
Actually this idea is already exposed in Section 4.3 of \cite{HK},
therefore instead of repeating a huge part of \cite{HK}, we suggest consulting \cite{HK}.

\bibliographystyle{plain}
\bibliography{moja}

\end{document}